\documentclass[12pt]{amsart}
\usepackage[english]{babel}
\usepackage[utf8x]{inputenc}
\usepackage{amsfonts}
\usepackage{amsmath}
\usepackage{longtable}
\usepackage{graphicx}
\usepackage{enumerate}
\usepackage{tikz-cd}
\usepackage{bbm}
\usepackage{mathtools}
\usepackage{amsthm}
\usepackage{amssymb}
\usepackage{url,smartref}
\usepackage{mathrsfs} 
\usepackage{tikz}
\usepackage{extpfeil,bm,fancyhdr,mathbbol}
\usepackage[top=1in, bottom=1.25in, left=1.25in, right=1.25in]{geometry}
\usetikzlibrary{shapes, arrows}
\tikzstyle{terminator} = [rectangle, draw, text centered, rounded corners, minimum height=2em]

\numberwithin{equation}{section}

\usepackage{multirow}
\usepackage{subfigure}
\usepackage{algorithm}
\usepackage{algorithmic}
\usepackage{comment,mathrsfs}

\title{On Modular Approach to Diophantine Equation $x^4-y^4=nz^p$ over Number Fields}
\author{Erman I{\c S}IK}

\address{University of Ottawa, Department of Mathematics and Statistics, STEM Complex, 150 Louis-Pasteur Pvt, Ottawa, ON, Canada}
\email{\url{eisik@uottawa.ca}}
\urladdr{\url{https://sites.google.com/view/erman-isik/}}

\newcommand{\mylabel}[2]{#2\def\@currentlabel{#2}\label{#1}}
\DeclareFontFamily{U}{wncy}{}
    \DeclareFontShape{U}{wncy}{m}{n}{<->wncyr10}{}
    \DeclareSymbolFont{mcy}{U}{wncy}{m}{n}
    \DeclareMathSymbol{\Sh}{\mathord}{mcy}{"58}

\theoremstyle{definition}
\newtheorem{definition}{Definition}[section]
\newtheorem{lemma}[definition]{Lemma}
\newtheorem{theorem}[definition]{Theorem}
\newtheorem{prop}[definition]{Proposition}
\newtheorem{cor}[definition]{Corollary}
\newtheorem{conj}[definition]{Conjecture}

\newtheorem{remark}[definition]{Remark}

\usepackage{enumitem, hyperref}
\makeatletter
\def\namedlabel#1#2{\begingroup
    #2%
    \def\@currentlabel{#2}%
    \phantomsection\label{#1}\endgroup
}
\makeatother

\DeclareFontFamily{U}{wncy}{}
    \DeclareFontShape{U}{wncy}{m}{n}{<->wncyr10}{}
    \DeclareSymbolFont{mcy}{U}{wncy}{m}{n}
    \DeclareMathSymbol{\Sh}{\mathord}{mcy}{"58}

\newcommand{\ccC}{\mathcal{C}}
\newcommand{\ccD}{\mathcal{D}}

\newcommand{\ccU}{\mathcal{U}}

\newcommand{\ccO}{\mathcal{O}}

\newcommand{\ccN}{\mathcal{N}}

\newcommand{\bbP}{{\mathbb P}}

\newcommand{\bbF}{{\mathbb F}}

\newcommand{\bbQ}{{\mathbb Q}}

\newcommand{\bbZ}{{\mathbb Z}}
\newcommand{\bbN}{{\mathbb N}}

\newcommand{\fraf}{{\mathfrak f}}

\newcommand{\frap}{{\mathfrak p}}
\newcommand{\fraq}{{\mathfrak q}}

\newcommand{\frakN}{{\mathfrak N}}

\newcommand{\frakP}{{\mathfrak P}}

\newcommand{\frakS}{{\mathfrak S}}

\newcommand{\gal}{{\rm Gal}}
\newcommand{\tr}{{\rm Tr}}
\newcommand{\norm}{{\rm Norm}}

\newcommand{\frob}{{\rm Frob}}

\newcommand{\cyc}{{\rm cyc}}

\newcommand{\aut}{{\rm Aut}}

\newcommand{\ord}{{\rm ord}}

\newcommand{\gl}{{\rm GL}}
\newcommand{\SL}{{\rm SL}}

\usepackage[OT2,T1]{fontenc}
\DeclareSymbolFont{cyrletters}{OT2}{wncyr}{m}{n}
\DeclareMathSymbol{\Sha}{\mathalpha}{cyrletters}{"58}

\newcommand{\OK}{{\mathcal{O}_K}}

\begin{document}
\maketitle

\begin{abstract}
    Recent results of Freitas, Kraus, {\c S}eng{\" u}n and Siksek give sufficient criteria for the asymptotic Fermat’s Last Theorem to hold over various number fields. 

\vspace{2mm}
    
    In this paper, we prove asymptotic results about the solutions of the Diophantine equation $x^4-y^4 = nz^p$ over various number fields using the modular method.  For instance, we prove that the asymptotic generalised Fermat Theorem for the equation $x^4-y^4 = 2^\alpha z^p$ holds for infinitely many quadratic number fields.
    
\end{abstract}

\section{Introduction}

 The study of \textit{generalized Fermat-type equations}, namely those of the form

\begin{equation*}
    Ax^p+By^q=Cz^r
\end{equation*}
for fixed non-zero coprime integers $A,B,C$, is one of the oldest and most widely studied topics in mathematics. In general, finding all solutions of such an equation is a difficult task, as there is no known general method that would allow us to produce solutions of a given Diophantine equation. 

\vspace{3mm}

 %In \cite{DG}, Darmon and Granville proved that, for fixed positive integers $A, B, C, p, q$ and $r$ with $1/p+1/q+1/r<1$, the generalized Fermat equation $Ax^p+By^q=Cz^r$ has only finitely many integer solutions $(x,y,z)$ such that $xyz\neq 0$.
 
 In \cite{wil95}, Andrew Wiles proved Fermat’s Last Theorem (FLT) using his celebrated modularity theorem for elliptic curves together with Ribet’s level lowering theorem \cite{ribet90}. Since then, several mathematicians have attempted to extend this approach to various Diophantine equations and number fields. The results known until $2020$ are summarised in \cite[Tables 1 and 2]{IKO20}.

 \vspace{3mm}

The strategy to tackle Diophantine equations via Wiles' approach, which is known as the ``modular method'', starts with an elliptic curve attached to a putative solution of the given equation. Then, using many celebrated results of the area, the problem can be reduced to one of the following: either computing newforms of a certain level, or computing all elliptic curves of a given conductor with particular information about the torsion subgroup and rational isogeny, or computing all solutions to an $S$-unit equation. It is important to note that none of these computations is straightforward in general. For instance, some fundamental theorems, such as the modularity of elliptic curves, that go into the proof are for the time being conjecture only. 

\vspace{3mm}

Recently, there has been important progress towards generalisation of the modular method over a larger class of number fields. In \cite{freitassiksek15}, using the modularity of elliptic curves over real quadratic fields, Freitas and Siksek proved the asymptotic FLT for infinitely many real quadratic fields. Namely, they showed that there is a constant $B_K$ depending only on the real quadratic field $K$, such that the only solutions to the equation $x^p + y^p = z^p$ with the exponent $p > B_K$ are the trivial ones. In \cite{ss18}, {\c S}eng{\" u}n and Siksek proved the asymptotic FLT for any number field $K$ by assuming two deep but standard conjectures from the Langlands program such as Serre's modularity conjecture. In \cite{dec16, KO20}, Deconinck and Kara-{\"O}zman extented  these results to the generalized Fermat-equation $Ax^p+By^p=Cz^p$ (see also \cite{OS22}). In \cite{IKO20,iko22}, I{\c s}{\i}k, Kara and {\"O}zman studied the solutions of the Fermat-type equations $x^p+y^p=z^2$ and $x^p+y^p=z^3$ over various number fields utilising a similar approach carried out in \cite{freitassiksek15,ss18}. In \cite{moc22}, Mocanu improved the results in \cite{IKO20,iko22} and proved similar results over totally real number fields. Following the ideas in \cite{nunof15} to construct the relevant Frey curve, Mocanu proved asymptotic results for the Fermat-equation $x^p+y^p=z^r$ over totally real number fields in \cite{Moc23}.

\vspace{3mm}

Let $K$ be a number field and $\ccO_{K}$ be its ring of integers.  For a prime number $p$ and a positive integer $n$ such that $\ord_\fraq(n)<p$ for all primes $\fraq$ of $K$ dividing $n$, consider the equation
\begin{equation*}\label{mainequation1}
x^4-y^4=nz^p,\quad x,y,z \in \ccO_{K}.
\end{equation*}

\vspace{3mm}

When $K=\bbQ$, in \cite{pow83}, Powell proved that the equation $x^4-y^4=nz^p$ has no integer solutions with $p\nmid xyz$. Terai-Osada \cite{to92} and Cao \cite{cao99} have extended this analysis to the equations $x^4 + dy^4 = z^p$ and $cx^4 + dy^4 = z^p$ respectively, proving that there are no certain types of solutions under certain assumptions. The methods used by these authors involve factoring the left-hand side of the equations over the appropriate quadratic field and applying a strategy analogous to Kummer’s work on Fermat’s Last Theorem.

\vspace{3mm}

The equation $x^4-y^4=nz^p$ has been studied in \cite{darmon93, dab07, ben21} using the modular method, again, over $\bbQ$. For instance, in \cite{darmon93}, Darmon proved that the equation $x^4-y^4=z^p$ has no non-trivial solution if $p\equiv 1 \pmod{4}$, or if $2\mid  z$. Note that the results in \cite{tor22} can be generalized to this equation over number fields other than $\bbQ$. In \cite{dab07}, Dabrowski proved that the equation
\begin{equation*}
    x^4-y^4=2^\alpha \ell^\beta z^p, \text{ for a fixed prime } \ell \text{ with } \ell\neq 2^k\pm1 \text{ with an integer } k,
\end{equation*}
has no solutions in coprime positive integers $x, y$ and suitably large prime exponent $p$. Some stronger results have been obtained in a number of papers, under the additional assumption that $y$ is prime; see, for example, \cite{ft10,savin09,yan13,adlt11}.

\vspace{3mm}

The main objective of the present article is to give asymptotic results for the equation $x^4-y^4=nz^p$ over various number fields.

\vspace{3mm}

\subsection{Statements of the main results}\label{statementsoftheresults}
Let $K$ be a number field and $\ccO_{K}$ be its ring of integers.  For a prime number $p$ and a positive integer $n$ such that $\ord_\fraq(n)<p$ for all primes $\fraq$ of $K$ dividing $n$, consider the equation
\begin{equation}\label{mainequation1}
x^4-y^4=nz^p,\quad x,y,z \in \ccO_{K}.
\end{equation}

\vspace{3mm}

 A solution $(a,b,c)$ to \eqref{mainequation1}  is
called \textbf{trivial} if $abc=0$, and called \textbf{primitive} if $a, b $ and $c$ satisfy $a\OK+b\OK+c\OK=\OK$.   Define the following set of primes of $K$
\begin{align*}
      S_K&=\{\frakP: \frakP \text{ is a prime of }K\text{ above }2\} \\
      S&=S_K\cup \{\fraq: \fraq \text{ is an odd prime of }K\text{ dividing }n\}.
\end{align*}

Consider the following hypothesis for the number field $K$:
\begin{itemize}
  \item[\namedlabel{itm:HNf}{(\textbf{H.Nf})}] Assume that $K$ is either totally real or an imaginary quadratic field such that the Mordell-Weil group $X_0(15)(K)$ is finite. In the latter case, assume that Conjecture \ref{conj2} holds for $K$.
\end{itemize}

\begin{remark}
   In order to run the modular approach, one needs the modularity theorems adapted to our setting (see \S \ref{overview} for the main steps of the strategy). When the number field $K$ is totally real, we will show that if $p$ is large enough the Frey curve is modular (see Proposition \ref{modularityoffreycurveovertotallyreal}). In general, due to the lack of the existence of modularity results, we can only prove our theorems by assuming the modularity conjecture (cf. Conjecture~\ref{conj1}). Additionally, we will use a special case of a fundamental conjecture from Langlands program (Conjecture~\ref{conj2}), which says that every weight $2$ newform (for ${\rm GL}_2$) over $K$ with integer Hecke eigenvalues has an associated elliptic curve over $K$ or a fake elliptic curve over $K$.
  
\end{remark}

\vspace{2mm}

\begin{theorem}\label{maintheoremA}
    Let $K$ be a number field, and fix a positive integer $n$. If $K$ does not satisfy \ref{itm:HNf}, we assume Conjecture~\ref{conj1} and Conjecture~\ref{conj2} hold for $K$. Suppose that there exits some distinguished prime $\widetilde{\frakP}\in S_K\subset S$ such that every solution $(\lambda,\mu)\in \left(\ccO_{S}^\times \right)^2$ to the $S$-unit equation
    \begin{equation}\label{unitequation}
        \lambda+\mu=1
    \end{equation}    
satisfies  $\max\{|\ord_{\widetilde{\frakP}}(\lambda)|, |\ord_{\widetilde{\frakP}}(\mu) |\}\leq 4\ord_{\widetilde{\frakP}}(2)$  and $\ord_{\widetilde\frakP}(n)\neq 4\ord_{\widetilde\frakP}(2)$. Then there is a constant $B_K$ depending only on $K$ such that, for $p > B_K$, the equation $x^4-y^4=nz^p$ does not have any non-trivial primitive solution $(a,b,c)$ in $K$ such that $\widetilde{\frakP} \mid c$. 
\end{theorem}

\vspace{3mm}

\begin{remark}
    We say that the constant $B_K$ is \textit{effective} when it is effectively computable. The main obstruction regarding the effectiveness of the bound $B_K$ comes from the modularity theorem. Although the Frey curve will be modular when we consider totally real number fields after enlarging the bound, the technique we will use makes the bound ineffective. However, it is important to note that this bound is effective when we restrict ourselves to real quadratic or cubic fields, or different layers of the cyclotomic $\bbZ_2$-extension of $\bbQ$ due to the results in \cite{fls15,dns20, tho19}.
\end{remark}

\vspace{3mm}
 
\begin{remark}
    As we can see from the statement, the proof of Theorem~\ref{maintheoremA} will depend on the solutions to $S$-unit equation \eqref{unitequation}. It is known that, for a given number field $K$ and a finite set of primes $S$, the $S$-unit equation \eqref{unitequation} has only finitely many solutions (see \cite{sie14}). Although there is no theoretical obstruction, it becomes practically challenging to find all the solutions when either $[K:\bbQ]$ or $|S|$ is too large.
\end{remark}

\subsubsection{Application to quadratic fields for the equation $x^4-y^4=2^\alpha\ell^\beta z^p$}

It follows from \cite[Lemma 7.5]{dec16} and \cite[Lemma 7.3]{KO20} that we are able to study the solutions of the $S$-unit equation \eqref{mainequation} when we focus on the case $n=2^\alpha$ or $2^\alpha\ell^\beta$ and certain quadratic fields.

\begin{cor}\label{corollaryoftheoremAquadraticfields}
    Let $d\in \bbZ$ be a squarefree integer, and $K=\bbQ(\sqrt{d})$. Assume that we have one of the following conditions:
    \begin{enumerate}[label=(\roman*)]
         \item If $d>0$, i.e. $K$ is a real quadratic field, then $d\geq 13$ and $d\equiv 5 \pmod{8}$.
        \item If $d<0$, i.e. $K$ is an imaginary quadratic field, then $|d|\geq 7$ and $d \equiv 2,3 \pmod 4$. If the Mordell-Weil group $X_0(15)(K)$ is finite assume Conjecture \ref{conj2}; otherwise, assume that Conjecture~\ref{conj1} and Conjecture~\ref{conj2} hold for $K$.
    \end{enumerate}
       
Let $\ell\geq 29$ be a prime satisfying $\ell \equiv 5 \pmod 8$ and $\left(\frac{d}{\ell} \right)=-1$. Then there is a constant $B_{K,\ell}$ depending only on $K$ and $\ell$ such that, for $p > B_{K,\ell}$, the equation $x^4-y^4=2^\alpha\ell^\beta z^p$  does not have any non-trivial primitive solution $(a,b,c)$ in $K$ such that $2\mid c$. Here, $\alpha\neq 4$ if $d\equiv5 \pmod{8}$, and $\alpha\neq 8$ if $d\equiv 2,3 \pmod{4}$.
\end{cor}

\subsubsection{Application to general fields for the equation $x^4-y^4=2^\alpha z^p$}

Imposing local constraints, we get the following results for general number fields.

\begin{cor}\label{maintheoremptotallyramifies}
     Let $K$ be a number field of degree $d$, and let $p \geq 5$ be a prime. If $K$ does not satisfy \ref{itm:HNf}, we assume Conjecture~\ref{conj1} and Conjecture~\ref{conj2} hold for $K$. Suppose also that
    \begin{enumerate}[label=(\roman*)]
        \item $\gcd(d,p-1)=1$;
        \item $2$ totally ramifies in $K$, say $2\OK=\frakP^d$;
        \item $p$ totally ramifies in $K$.
    \end{enumerate}
    Then there is a constant $B_K$ depending only on $K$ such that, for $p > B_K$, the equation $x^4-y^4=2^\alpha z^p$ (with $\alpha\neq 4d$) does not have any non-trivial primitive solution $(a,b,c)$ in $K$ such that $\frakP \mid c$.
\end{cor}

\vspace{3mm}

\begin{cor}\label{maintheorem3totallysplits}
   Let $K$ be a number field of degree $d$. If $K$ does not satisfy \ref{itm:HNf}, we assume Conjecture~\ref{conj1} and Conjecture~\ref{conj2} hold for $K$. Suppose also that
   \begin{enumerate}[label=(\roman*)]
       \item $d$ is odd;
       \item $2$ totally ramifies in $K$, say $2\OK=\frakP^d$;
       \item $3$ totally splits in $K$.
   \end{enumerate}
 Then there is a constant $B_K$ depending only on $K$ such that, for $p > B_K$, the equation $x^4-y^4=2^\alpha z^p$ (with $\alpha\neq 4d$) does not have any non-trivial primitive solution $(a,b,c)$ in $K$ such that $\frakP \mid c$.
\end{cor}

\subsubsection{Application to quadratic fields for the equation $x^4-y^4=2^\alpha z^p$}

\begin{cor}\label{realquadraticwithequiv24}
      Let $q>73$ be a prime such that $q \equiv 1 \pmod{24}$, and $K = \bbQ(\sqrt{q})$. Then there is a constant $B_{K}$ depending only on $K$ such that, for $p > B_K$, the equation $x^4-y^4=2^\alpha z^p$ ( with $\alpha\neq 4$) does not have any non-trivial primitive solution $(a,b,c)$ in $K$ such that $2\mid c$. 
\end{cor}

\subsubsection{Density theorems} The results in \cite{freitassiksek15,KO20} show that there is a bound (effective for the real quadratic fields) such that the equation $x^4-y^4=2^\alpha z^p$ does not have any primitive solution $(a,b,c)$ such that $2\mid c$ for $5/6$ of real quadratic fields (and $5/6$ of imaginary quadratic fields assuming Conejecture~\ref{conj1} and Conjecture~\ref{conj2}).

\vspace{3mm}

Following the notation in \cite{freitassiksek15,KO20}, we define the set
\begin{equation*}
    \bbN^{\rm sf}=\{d\geq 2: d\; \text{a squarefree integer}\}.
\end{equation*}
 Then the elements of $\bbN^{\rm sf}$ are in bijection with real quadratic fields $\bbQ(\sqrt{d})$ (resp. with imaginary quadratic fields $\bbQ(\sqrt{-d})$ except for $\bbQ(i)$)  via the canonical map $d\leftrightarrow \bbQ(\sqrt{d})$ (resp.  via $d\leftrightarrow \bbQ(\sqrt{-d})$). For a subset $\ccU\subset \bbN^{\rm sf}$, define the \textit{relative density of $\ccU$ in $\bbN^{\rm sf}$} by
\begin{equation*}
    \delta_{\rm rel}(\ccU):=\lim_{X\to \infty}\dfrac{\sharp\{d\in\ccU: d\leq X \}}{\sharp\{d\in\bbN^{\rm sf}: d\leq X \}},
\end{equation*}
provided the limit exists. Note that $\{(-1,2),(2,-1),(1/2,1/2)\}$ is a set of solutions to the $S$-unit equation \eqref{unitequation}. Following \cite{freitassiksek15} these solutions will be called \textit{irrelevant} solutions, while other solutions will be called \textit{relevant} solutions. Let
\begin{align*}
    \ccC&=\{d\in \bbN^{\rm sf}: \text{the }S\text{-unit}  \text{ equation }\eqref{unitequation} \text{ has no relevant solutions in }\bbQ(\sqrt{d})\},\\
    \ccC'&=\{d\in \bbN^{\rm sf}: \text{the }S\text{-unit}  \text{ equation }\eqref{unitequation} \text{ has no relevant solutions in }\bbQ(\sqrt{-d})\},
    \end{align*}
\begin{align*}
    \ccD&=\{d\in\ccC: d \not\equiv 5 \pmod{8} \}, \\
    \ccD'&=\{d\in\ccC: -d \not\equiv 5 \pmod{8} \}.
\end{align*}

\begin{theorem}[Theorem 4 in \cite{freitassiksek15} and Theorem 9.4 in \cite{KO20}]\label{densitytheorem}
    Let $\ccC$ and $\ccD$ be above. Then we have
    \begin{equation*}
        \delta_{\rm rel}(\ccC)=\delta_{\rm rel}(\ccC')=1 \text{ and }  \delta_{\rm rel}(\ccD)=\delta_{\rm rel}(\ccD')=5/6.
    \end{equation*}
\end{theorem}

Combining Theorem~\ref{densitytheorem} and Theorem~\ref{maintheoremA} (applied for $n=2^\alpha$) we deduce the following result:

\begin{cor}\label{densitytheoremforourcase}
    If $d\in \ccD$ (resp. $d\in\ccD'$), let $K=\bbQ(\sqrt{d})$ (resp. $K=\bbQ(\sqrt{-d})$), then there is an bound $B_K$ such that the equation $x^4-y^4=2^\alpha z^p$ (with $\alpha\neq 4$) does not have any solutions $(a,b,c)$ such that $2\mid c$ over $K$ (assume Conjecture~\ref{conj2} if the Mordell-Weil group $X_0(15)(K)$ is finite; otherwise, we assume Conjecture \ref{conj1} and Conjecture \ref{conj2} in the imaginary quadratic case).
\end{cor}

\vspace{3mm}

For a given number field $K$, denote the narrow class number by $h_K^+$. Imposing restrictions on the narrow class number and the prime ideal lying above $2$, we obtain:

\begin{theorem}\label{maintheoremB}
      Let $K$ be a number field with $2\nmid h^+_K$, and assume that there is only one prime ideal $\frakP$ of $K$ lying above $2$. If $K$ does not satisfy \ref{itm:HNf}, we assume Conjecture~\ref{conj1} and Conjecture~\ref{conj2} hold for $K$. Then there is a constant $B_K$ depending only on $K$ such that, for $p > B_K$, the equation $x^4-y^4=2^\alpha z^p$ (with $\alpha\neq 4\ord_{\frakP}(2)$) does not have any non-trivial primitive solution $(a,b,c)$ in $K$ such that $\frakP \mid  c$.
\end{theorem}

More precisely, let $K=\bbQ(\sqrt{\pm q})$ be a quadratic field, where $q$ is a rational prime. By \cite[Proposition 1.3.2]{Lem10}, we have $2\nmid h_K^+$. The following theorems are immediate consequences of Theorem~\ref{maintheoremB}. 

\vspace{3mm}

\begin{cor}\label{corollaryoftheoremBquadraticfields}
    Let $q$ be a rational prime. Assume that we have one of the following conditions:
    \begin{enumerate}[label=(\roman*)]
        \item If $q$  satisfies $q\equiv 5 \pmod{8}$ or $q\equiv 3 \pmod 4$, set $K = \bbQ(\sqrt{q})$.
        \item If $q$ satisfies $q\equiv 3 \pmod{8}$ or $q\equiv 1 \pmod 4$, set $K = \bbQ(\sqrt{-q})$. If the Mordell-Weil group $X_0(15)(K)$ is finite assume Conjecture~\ref{conj2}; otherwise, we assume Conjecture \ref{conj1} and Conjecture \ref{conj2}.
    \end{enumerate}
Then there is a constant $B_{K}$ depending only on $K$ such that, for $p > B_K$, the equation $x^4-y^4=2^\alpha z^p$ ($\alpha\neq 4$ or $8$ depending on the behavior of $2$ in $K$) does not have any non-trivial primitive solution $(a,b,c)$ in $K$ such that $\frakP\mid  c$. 
\end{cor}

Let $\bbQ_{r,2}:=\bbQ(\zeta_{2^{r+2}})^+$ denote the $r$-th layer of the cyclotomic $\bbZ_2$-extension $\bbQ_{\infty,2}$ of $\bbQ$, for $r\geq 1$. Note that $\bbQ_{r,2}$ is a totally real field with odd narrow class number (see \cite[II]{iwa56} and \cite[Corollary 2.4]{FKS}) and $2$ is totally ramified in $\bbQ_{r,2}$. Denote the prime ideal of $\bbQ_{r,2}$ above $2$ by $\frakP_{r,2}$. Moreover, it follows from \cite{tho19} that elliptic curves over $\bbQ_{\infty,2}$ are modular. We then immediately deduce the following result:

\begin{cor}\label{corollaryoftheoremBcyclotomicZ2}
      There is an effective constant $B_{\bbQ_{r,2}}$ depending only on $\bbQ_{r,2}$ such that, for $p > B_{\bbQ_{r,2}}$, the equation $x^4-y^4=2^\alpha z^p$ ($\alpha\neq 2^{r}$) does not have any non-trivial primitive solution $(a,b,c)$ in $K$ such that $\frakP_{r,2}\mid  c$.
\end{cor}

\subsection{A brief overview of the strategy}\label{overview}
Let $K$ be a number field and let $G_K:=\gal(\overline{K}/K)$ denote its absolute Galois group. Let $\widetilde\frakP$ be a fixed prime of $K$ lying above $2$. We summarise the modular method to tackle the Diophantine equation $x^4-y^4=nz^p$ over $K$ in the following steps:
\begin{description}
    \item[Frey curve] To a putative solution $(a,b,c)\in \ccO_K^3$ to the equation~\eqref{mainequation} with $\widetilde\frakP\mid c$, we attach a Frey elliptic curve $E/K$ with full $2$-torsion over $K$, having semistable reduction outside the primes dividing $2n$, and potentially multiplicative reduction at $\widetilde\frakP$.
    
    \item[Modularity and Level Lowering] For totally real number fields, we show that the Frey curve $E$ is modular. In general, we assume some hypothesis to get an automorphic form $\fraf$ over $K$, which lies in a space that does not depend on the solution $(a,b,c)$ and the exponent $p$, such that
    \begin{equation*}
        \overline{\rho}_{E,p}\sim \overline{\rho}_{\fraf,\varpi},
    \end{equation*}
   where  $\varpi$ is some prime $\bbQ_\fraf$ above $p$.
    
    \item[Eichler-Shimura] In this step, we obtain (assuming $p$ large enough) an elliptic curve $E_\fraf/K$ with full $2$-torsion over $K$, having potentially good reduction outside the primes of $K$ dividing $2n$ and potentially multiplicative reduction at $\widetilde\frakP$.
    
    \item[Contradiction] We first establish the relation between elliptic curves with full $2$-torsion over $K$, having potentially good reduction outside $S$, and solutions of the $S$-unit equation \eqref{unitequation}, and then obtain a contradiction using the assumptions.  For Theorem \ref{maintheoremB}, we show that the elliptic curve we obtained in the previous step cannot exist.
\end{description}

\subsection{Acknowledgements} The author is supported by the NSERC Discovery Grants Program RGPIN-2020-04259 and RGPAS-2020-00096.

\section{Frey Curve associated to the equation $x^4-y^4=nz^p$}

In this section, we collect some facts related to the Frey curve associated with a putative solution to the equation \eqref{mainequation1}. 

\vspace{3mm}

We use $\fraq$ for an arbitrary prime of $K$, and $G_\fraq$ and $I_\fraq$ for the decomposition and inertia subgroups of $G_K$ at $\fraq$, respectively. Let $E/K$ be an elliptic curve, and let $\overline{\rho}_{E,p}$ denote the residual Galois representation attached to $E$
\begin{equation*}
    \overline{\rho}_{E,p}: G_K \longrightarrow \aut(E[p])\cong \gl_2(\bbF_p).
\end{equation*}

 Assume that $(a,b,c)\in \ccO_K ^3$ is a non-trivial primitive solution to the equation
\begin{equation}\label{mainequation}
    x^4-y^4=nz^p
\end{equation}
with the prime exponent $p$ such that $\ord_{\fraq}(n)<p$ for any prime $\fraq$ of $K$. Following the arguments in \cite{darmon93,dab07,ben21}  we define the following elements in $\ccO_K$
\begin{align*}
    A &= (a+b)^2 = a^2+2ab+b^2,\\
     B &= (a-b)^2 = a^2-2ab+b^2,\\
     C &= a^2+b^2,
\end{align*}
which  satisfy the equation
\begin{equation}\label{eqution2}
    A+B-2C=0.
\end{equation}

Consider the Frey curve
\begin{equation}\label{Freycurve}
    E:=E_{(A,B,C)}: y^2=x(x+A)(x-B).
\end{equation}

When we expand the right hand side of \eqref{Freycurve}, the equation for $E$ becomes:
\begin{equation}\label{freycurve2}
    E: y^2=x^3+4abx^2-(a^2-b^2)^2x.
\end{equation}

The invariants associated to the Frey curve $E$ are given by:
\begin{align*}
     \Delta_E&=2^6n^2c^{2p}(a^2-b^2)^2=2^6(ABC)^2;\\
     c_4&=2^4(a^2+3b^2)(3a^2+b^2)=2^4(AB-2AC-2BC);\\
     c_6&=   -2^7\left(2^5a^2b^2+3(a^2-b^2)^2\right);\\
    j_E&=2^6\frac{(a^2+3b^2)^3(3a^2+b^2)^3}{n^2c^{2p}(a^2-b^2)^2}=2^6\frac{(2AC+2BC-AB)^3}{ (ABC)^2}.
\end{align*}

\begin{lemma}
    Let $p>3$ be a rational prime. Then the equation~\eqref{mainequation} has infinitely many non-primitive solutions.
\end{lemma}
\begin{proof}
  Let $a,b \in \ccO_K$ be arbitrary elements. Say $u:=na$, $v:=nb$ and $r=n^3(a^4-b^4)$ so that $u^4-v^4=nr$. It follows that we can find arbitrary elements $u,v\in \ccO_K$ such that $u^4-v^4=nr$ for some $r\in\ccO_K$. If $p \equiv 1 \pmod 4$, then $(ur^{\frac{p-1}{4}},vr^{\frac{p-1}{4}}, r)$ is a non-primitive solution to the equation~\eqref{mainequation}. If $p \equiv 3 \pmod 4$, then $(ur^{\frac{3p-1}{4}},vr^{\frac{3p-1}{4}}, r^3)$ is a non-primitive solution to the equation~\eqref{mainequation}.
\end{proof}

Therefore, we only need to consider the primitive solutions to the equation~\eqref{mainequation}. The following remark will be used when we need to consider the image of the inertia subgroup $I_{\widetilde\frakP}$ under the Galois representation $\overline{\rho}_{E,p}$ attached to $E$.

\begin{remark}\label{remarkonvaluations}
    Fix $\widetilde\frakP \in S_K$. Let $(a,b,c)\in\ccO_K^3$ be a primitive solution to the equation \eqref{mainequation}, where $p> 2\ord_{\widetilde\frakP}(2)$ and $\widetilde \frakP\mid c$. Since $(a^2+b^2)(a^2-b^2)=nc^p$ and $a^2+b^2=(a^2-b^2)+2b^2$, we have $\widetilde\frakP\mid \gcd(a^2+b^2,a^2-b^2)$. Further, if $\ord_{\widetilde\frakP}(a^2-b^2)>\ord_{\widetilde\frakP}(2)$, then $\ord_{\widetilde\frakP}(a^2+b^2)=\ord_{\widetilde\frakP}(2)$ as $\widetilde\frakP\nmid b$. On the other hand, if $\ord_{\widetilde\frakP}(a^2-b^2)\leq \ord_{\widetilde\frakP}(2)$, then we have $\ord_{\widetilde\frakP}(a^2+b^2)\geq \ord_{\widetilde\frakP}(a^2-b^2)$. Thus, we deduce that $\ord_{\widetilde\frakP}(a^2+b^2)> \ord_{\widetilde\frakP}(2)$, since otherwise we would have
    \begin{equation*}
        p \leq p\ord_{\widetilde\frakP}(c) \leq 2 \ord_{\widetilde\frakP}(a^2+b^2)\leq 2\ord_{\widetilde\frakP}(2).
    \end{equation*}
    Hence, $\min\{\ord_{\widetilde\frakP}(a^2+b^2),\ord_{\widetilde\frakP}(a^2-b^2) \}> \ord_{\widetilde\frakP}(2)$. We then conclude that, interchanging $b^2$ by $-b^2$ if necessary, we may assume that
    \begin{equation*}
        \ord_{\widetilde\frakP}(a^2+b^2)=\ord_{\widetilde\frakP}(2) \text{ and } \ord_{\widetilde\frakP}(a^2-b^2)=p\ord_{\widetilde\frakP}(c)+\ord_{\widetilde\frakP}(n)-\ord_{\widetilde\frakP}(2)
    \end{equation*}

\end{remark}

Let $\ccN_E$ denote the conductor of $E$, and define
\begin{equation}\label{reducedconductor}
    \ccN_p=\ccN_E\bigg/ {\displaystyle\prod_{\substack{\fraq\mid \mid \ccN_E\\p\mid \ord_\fraq(\Delta_\fraq)}}}\fraq,
\end{equation}
where $\Delta_\fraq$ denotes the discriminant of a local minimal model of $E$ at $\fraq$. The following lemma collects the facts related to the Frey curve $E$.

\begin{lemma}\label{conductoroffreycurve} 
	The Frey curve $E$ is semistable away $S_K$ and has a full $2$-torsion over $K$.
	The determinant of $\overline{\rho}_{E,p}$ is the mod $p$ cyclotomic character. The Galois representation $\overline{\rho}_{E,p}$ is finite flat at every prime $\frap$ of $K$ that lies above $p$. 
	Moreover, the conductor $\mathcal{N}_E$ attached to the Frey curve $E$ is given by
 \begin{equation*}
     \mathcal{N}_E=\prod_{\substack{\frakP\in S_K}}\frakP^{r_\frakP}\prod_{\substack{\mathfrak{q}\mid  cn,\\ \mathfrak{q} \nmid 2}}\mathfrak{q} \text{ and } \ccN_p=\prod_{\substack{\frakP\in S_K}}\frakP^{r'_\frakP}\prod_{\substack{\fraq\mid n \\ \mathfrak{q} \nmid 2}}\fraq,
 \end{equation*}
where $0\leq r'_\frakP\leq r_\frakP\leq 2+6\ord_\frakP(2)$.	 

\vspace{3mm}

	The Serre conductor $\frakN_E$, which is the prime-to-$p$ part of the Artin conductor of $\overline{\rho}_{E,p}$, is supported on $S_K\cup\{\fraq:\fraq\mid n \text{ odd}\}$ and belongs to a finite set depending only on the field $K$ and $n$.  
\end{lemma}
\begin{proof}
    Recall that the invariants $c_4$ and $\Delta_E$ are given by
    \begin{equation*}
         c_4=2^4(a^2+3b^2)(3a^2+b^2)\;\text{and}\; \Delta_E= 2^6n^2c^{2p}(a^2-b^2)^2.
    \end{equation*}    
    If $\fraq\not\in S_K$ does not divide $cn$, then $\ord_\fraq(\Delta_E)=0$. Therefore, the model is minimal and $E$ has good reduction at $\fraq$. Note that if $\fraq\mid  \gcd(a^2+b^2,a^2-b^2)$, then $\fraq$ divides $2a^2$ and $2b^2$, so $\fraq\mid  2$ as $(a,b,c)$ is primitive. Suppose that $\fraq\not\in S_K $ divides $\Delta_E$, which implies that $cn$ is divisible by $\fraq$. Therefore, $c_4=2^4(a^2+3b^2)(3a^2+b^2)$ is not divisible by $\fraq$, i.e. $\ord_\fraq(c_4)=0$. Thus, the given model is minimal and $E$ has multiplicative reduction at $\fraq$. We have
    \begin{equation*}
        \ord_\fraq(\Delta_E)=2p\ord_\fraq(c)+2\ord_\fraq(n)+2\ord_\fraq(a^2-b^2).   
    \end{equation*}
    Let $\fraq\not\in S_K$ such that $\fraq\mid cn$. Then we have the following cases:
    \begin{itemize}
        \item If $\fraq\mid c$ and $\fraq\nmid n$, then since $\fraq\not\in S_K$ we have $p\mid \ord_\fraq(a^2-b^2)$, so $p\mid \ord_\fraq(\Delta_E)$.
        \item  If $\fraq\nmid c$ and $\fraq\mid  n$, then $\ord_{\fraq}(n)=\ord_{\fraq}(a^2-b^2)<p$. Therefore, we have $p\nmid\ord_\fraq(\Delta_E)$.
        \item  If $\fraq\mid  c$ and $\fraq\mid  n$, then we have $p\nmid\ord_\fraq(\Delta_E)$.
    \end{itemize}
   Thus, if $\fraq\not \in S_K$ satisfies $\fraq\mid n$, then we have $p\nmid\ord_\fraq(\Delta_E)$. It follows from \cite{serre76} that $\overline{\rho}_{E,p}$ is finite flat $\fraq$ if $\fraq$ lies above $p$ as we assumed $p\nmid n$. We can also deduce that $\overline{\rho}_{E,p}$ is unramified at $\fraq$ for odd primes $\fraq$ of $K$ such that $\fraq\nmid np$.

\vspace{3mm}

    For each $\frakP\in S_K$, the bounds on the exponents $r_\frakP$ and $r'_\frakP$ follow from \cite[Theorem IV.10.4]{Sil94}. As the Serre conductor $\frakN_E$ divides $\ccN_E$ and is divisible only by the primes $S_K\cup\{\fraq:\fraq\mid n \text{ odd}\}$, there can be only finitely many Serre conductors and they only depend on $K$ and $n$.

\vspace{3mm}
    
    The statement concerning the determinant is a well-known consequence of Weil pairing attached to elliptic curves. It immediately follows from the equation \eqref{Freycurve} that the Frey curve $E$ has full $2$-torsion over $K$.
    \end{proof}

\section{Modularity of the Frey Curve $E$}
If $E$ is an elliptic curve over a number field $K$, we say that $E$ is modular if there is an isomorphism of a compatible system of Galois representations
\begin{equation*}
    \rho_{E,p}\sim \rho_{\fraf,\varpi},
\end{equation*}
where $\fraf$ is an automorphic form over $K$ of weight $2$ with coefficient field $\bbQ_\fraf$ and $\varpi$ is a prime in $\bbQ_\fraf$ lying above $p$

\subsection{Totally real case} In the totally real case, $\fraf$ comes from a Hilbert eigenform of parallel weight $2$ over $K$. The following result is due to Freitas, Le Hung and Siksek \cite[Thereoms 1 and 5]{fls15}:
\begin{theorem}\label{modularityovertotallyreal}
    Let $K$ be a totally real number field. There are at most finitely many $\overline{K}$-isomorphism classes of non-modular elliptic curves $E$ over $K$. Moreover, if $K$ is real quadratic, then all elliptic curves over $K$ are modular.
\end{theorem}

Moreover, Derickx, Najman and Siksek proved the following result in \cite{dns20}:
\begin{theorem}\label{modularitytotallyrealcubic}
    Let $K$ be a totally real cubic number field and $E$ be an elliptic curve over $K$. Then $E$ is modular.
\end{theorem}

Further results have been obtained for quartic totally real fields by Josha Box \cite{box22}.

\begin{theorem}[Theorem 1.1, \cite{box22}]
    Let $E$ be an elliptic curve over a totally real quartic number field not containing a square root of $5$. Then $E$ is modular.
\end{theorem}

\begin{prop}\label{modularityoffreycurveovertotallyreal}
    Let $K$ be a totally real number field, and fix $\widetilde\frakP\in S_K$. There is some constant $A_K$, depending only on $K$, such that for any non-trivial primitive solution $(a,b,c)$ to the equation \eqref{mainequation} with and $\widetilde\frakP\mid c$ and the exponent $p>A_K$, the Frey curve $E$ is modular.
\end{prop}
\begin{proof}
    It follows from Theorem~\ref{modularityovertotallyreal} that there are at most finitely many possible $\overline{K}$-isomorphism classes of elliptic curves over $K$ which are non modular. Let $j_1,\dots, j_n \in K$ denote the $j$-invariants of these classes. Define $\lambda:=2^2\frac{a^2-b^2}{a^2+b^2}$ so that the $j$-invariant of $E$ is
    \begin{equation*}
        j(\lambda)=2^2(64-\lambda^2)^3\lambda^{-4}.
    \end{equation*}

We can assume that $\lambda\not\in \bbQ$, since these $\lambda$ would lead to $j(\lambda)\in \bbQ$, and it is known that all rational elliptic curves are modular. Each equation $j(\lambda)=j_i$ has at most six solutions in $K$. Thus there are values $\lambda_1,\dots, \lambda_m\in K$ (where $m\leq 6n$) such that if $\lambda\neq \lambda_k$ for all $k$, then the elliptic curve $E$ with the $j$-invariant $j(\lambda)$ is modular. If $\lambda=\lambda_k$ for some $k$, then, using $\ord_{\widetilde\frakP}(a^2+b^2)=\ord_{\widetilde\frakP}(2)$ and $\ord_{\widetilde\frakP}(a^2-b^2)=p\ord_{\widetilde\frakP}(c)+\ord_{\widetilde\frakP}(n)-\ord_{\widetilde\frakP}(2)$ (see Remark~\ref{remarkonvaluations}), we deduce that
\begin{equation*}
    \ord_{\widetilde\frakP}(\lambda_k)= p\ord_{\widetilde\frakP}(c)+\ord_{\widetilde\frakP}(n)\geq p.
\end{equation*}
As $\lambda_k$ is fixed, it gives an upper bound on $p$ for each $k$, and by taking the maximum of these bounds we get $A_K$.
\end{proof}

\begin{remark}
    The finiteness part of Theorem~\ref{modularityovertotallyreal} relies on Falting's Theorem \cite{falt83}. The bound in Falting's Theorem is ineffective, and so is the bound $A_K$. Note that if $K$ is quadratic or cubic we get $A_K=0$.
\end{remark}

\subsection{Imaginary quadratic case} In a recent article \cite{cn23}, Caraiani and Newton prove the modularity theorem for elliptic curves over $K$, where $K$ runs over infinitely many imaginary quadratic fields, including $K=\bbQ(\sqrt{-d})$ for $d =1, 2, 3, 5$.

\begin{theorem}[Thereom 1.1, \cite{cn23}]\label{modularityovertotallyimaginaryquadratic}
    Let $K$ be an imaginary quadratic field such that the Mordell-Weil group $X_0(15)(K)$ is finite. Then every elliptic curve $E/K$ is modular.
\end{theorem}

\subsection{General case} Let $K$ be an arbitrary number field (admitting possibly a complex embedding) with the ring of integers $\OK$, and let $\frakN$ be an ideal of $\OK$. 
  For notations, relevant definitions a detailed discussion concerning complex and mod $p$ eigenforms over $K$, we refer the reader to  \cite[Sections 2 and 3]{FKS}. 

  \vspace{3mm}
  
 The following conjecture is a special case of Serre's modularity conjecture over number fields.

\begin{conj}[\cite{FKS}, Conjecture 4.1]\label{conj1}  Let $\overline{\rho}:G_K\rightarrow GL_2(\overline{\mathbb{F}}_p)$ be an odd, irreducible, continuous representation with Serre conductor $\frakN$ (prime-to-$p$ part of its Artin conductor) 
	and such that $\det(\overline{\rho})=\chi_p$ is the mod $p$ cyclotomic character.   
	Assume that $p$ is unramified in $K$ and that $\overline{\rho}\mid _{G_{\bbQ_\frap}}$ arises from a finite-flat group scheme over $\ccO_{\bbQ_{\frap}}$ for every prime $\frap\mid p$.  Then there is a weight two, mod $p$ eigenform $\theta$ over $K$ of level $\frakN$ such that for all primes $\fraq$ coprime to $p\frakN$, we have
	\[
	\tr(\overline{\rho}(\frob_{\fraq}))=\theta(T_{\fraq}),
	\]
	where $T_{\fraq}$ denotes the Hecke operator at $\fraq$.
\end{conj}

\vspace{3mm}

\begin{remark}
   Given a number field $K$, we obtain a \emph{complex conjugation} for every real embedding 
$\sigma: K \hookrightarrow \mathbb R$ and every extension $\widetilde{\sigma}: \overline{K} \hookrightarrow \mathbb C$ of 
$\sigma$ as $\widetilde{\sigma}^{-1}\iota \widetilde{\sigma} \in G_K$ where $\iota$ is the usual complex conjugation. Recall that a representation $\overline{\rho}: G_K \rightarrow \gl_2(\overline{\mathbb F}_p)$ is \emph{odd} if the determinant of every complex  conjugation is $-1$.  If the number field $K$ has no real embeddings, then we immediately say that $\overline{\rho}$ is odd. It follows from Lemma~\ref{conductoroffreycurve} that $\overline{\rho}_{E,p}$ is odd.
 
\end{remark}

\section{Image of Inertia}

Let $K$ be a number field, and let $\fraq$ be a prime of $K$. In this section, we gather information about the image of inertia $\overline{\rho}_{E,p}(I_{\widetilde\frakP})$ for certain primes $\widetilde\frakP\in S_K$. This is a crucial step in controlling the behaviour of the newform obtained by level lowering at $\widetilde\frakP\in S_K$.

\begin{lemma}[\cite{freitassiksek15}, Lemma 3.4]\label{imageofinertia}
	Let $E$ be an elliptic curve over $K$ with $j$-invariant $j_E$.  Let $p \geq 5$ and $\mathfrak{q}\nmid p$ be a prime  
	of $K$. Then $p\mid \#\overline{\rho}_{E,p}(I_{\mathfrak{q}})$ if and only if $E$ has potentially multiplicative
	reduction at $\mathfrak{q}$ (i.e. $\ord_{\mathfrak{q}}(j_E)<0$) and $p \nmid \ord_{\mathfrak{q}}(j_E)$.
\end{lemma}   
By using the previous result we obtain:

\begin{lemma}\label{potmultred}
	Fix $\widetilde\frakP\in S_K$ and let $(a,b,c)\in \ccO_K^3$ be a non-trivial primitive solution to the equation $x^4-y^4=nz^p$ with $\widetilde\frakP\mid c$ and exponent $p>2\ord_{\widetilde\frakP}(2)$ such that $\ord_{\fraq}(n)<p$ for any prime $\fraq$ of $K$. Let $E$ be the Frey curve as in \eqref{Freycurve} associated to $(a,b,c)$, and write $j_E$ for its $j$-invariant. Then $E$ has potentially multiplicative reduction at $\widetilde\frakP$ and $p\mid \#\overline{\rho}_{E,p}(I_{\widetilde\frakP})$, where $I_{\widetilde\frakP}$ denotes an inertia subgroup of $G_{K}$ at $\widetilde\frakP$.
\end{lemma}  
\begin{proof}
    Let $\widetilde\frakP\in S_K$ with $\ord_{\widetilde\frakP}(c)=k>0$. Then we have 
    \begin{align*}
         \ord_{\widetilde\frakP}(j_E)&=6\ord_{\widetilde\frakP}(2)-2pk-2\ord_{\widetilde\frakP}(n)-2\ord_{\widetilde\frakP}(a^2-b^2)\\
         &=6\ord_{\widetilde\frakP}(2)-2pk-2\ord_{\widetilde\frakP}(n)-2(pk-\ord_{\widetilde\frakP}(2))\\
         &=8\ord_{\widetilde\frakP}(2)-4pk-2\ord_{\widetilde\frakP}(n).
    \end{align*}
    Since $p>2\ord_{\widetilde\frakP}(2)$ and $\ord_{\widetilde\frakP}(n)<p$, we have $E$ has potentially multiplicative reduction at $\widetilde\frakP$ (i.e. $\ord_{\widetilde\frakP}(j_E)<0$), and $p\nmid\ord_{\widetilde\frakP}(j_E)$ (as we assumed $\ord_{\widetilde\frakP}(n)\neq 4\ord_{\widetilde\frakP}(2)$). It then follows from Lemma~\ref{potmultred} that $p\mid \#\overline{\rho}_{E,p}(I_{\widetilde\frakP})$.
\end{proof}

\section{Irreducibilty of the associated Galois representation $\overline{\rho}_{E,p}$}

 The following well-known result about subgroups of $\gl_2(\mathbb F_p)$ will be used to prove the irreducibility of the Galois representation $\overline{\rho}_{E,p}$.

\begin{theorem}\label{subgroups} Let $E$ be an  elliptic curve over a number field $K$ of degree $d$ and let $G \leq \gl_2(\mathbb F_p)$ be the
	image of the mod $p$ Galois representation $\overline{\rho}_{E,p}$ of $E$.
	Then the following holds:
	\begin{itemize}
		\item if $p \mid  \#G$ then either $\overline{\rho}_{E,p}$ is reducible or $G$ contains ${\rm SL}_2(\mathbb F_p)$. In the latter case, we deduce that $\overline{\rho}_{E,p}$ is absolutely irreducible. 
		\item if $p \nmid \#G$ and $p > 15 d +1$ then $G$ is contained in a Cartan subgroup or $G$ is contained in the normalizer of a Cartan subgroup but not the Cartan subgroup itself.
		
	\end{itemize}
\end{theorem}
\begin{proof}
	For the proof, the main reference is \cite[Lemma 2]{SD}. The version above including the proof of the second part is from \cite[Propositions 2.3  and 2.6]{localglobal}.
\end{proof}

The following is Proposition 6.1 from \c{S}eng\"{u}n and Siksek \cite{ss18}. We include its statement for the convenience of the reader but we will omit its proof and refer to \cite{ss18} instead.

\vspace{3mm}

\begin{prop}\label{irredpp3}
	Let $L$ be a Galois number field and let $\fraq$ be a prime of $L$. There is a constant $B_{L,\fraq}$ such that the following is true. Let $p > B_{L, \fraq}$ be  a rational prime. Let $E/L$ be an elliptic curve that is semistable at all $\frap \mid  p$ and has potentially multiplicative reduction at $\fraq$. Then $\overline{\rho}_{E,p}$ is irreducible. 	
\end{prop}

\begin{cor}\label{galrepsurjective}
	Let $K$ be a number field, and fix $\widetilde\frakP\in S_K$.  There is a constant $C_K$ such that if $p>C_K$ and 
	$(a,b,c)$ is a non-trivial primitive solution to the Fermat equation  $x^4-y^4=nz^p$ with $\widetilde\frakP\mid c$ and exponent $p$ such that $\ord_{\fraq}(n)<p$ for any prime $\fraq$ of $K$, then  the Galois representation $\overline{\rho}_{E,p}$ is surjective.
\end{cor}
\begin{proof}
    It follows Lemma \ref{potmultred} that if $p>2\ord_{\widetilde\frakP}(2)$, the Frey curve $E$ has potentially multiplicative reduction at $\widetilde\frakP$.  Also, $E$ is semistable away from $S_K$ by Lemma \ref{conductoroffreycurve}.  Let $L$ be the Galois closure of $K$, and let $\widetilde\fraq$ be a prime	of $L$ above $\widetilde\frakP$.  Now, by applying Proposition \ref{irredpp3}, we get a constant $B_{L,\widetilde\fraq}$ such that $\overline{\rho}_{E,p}$ is irreducible whenever $p>B_{L,\widetilde\fraq}$. Note that there are only finitely many choices of $\widetilde\fraq$ in $L$ dividing $\widetilde\frakP$, and $L$ only depends on $K$.  Hence, we can obtain a constant depending only on $K$ and we denote it by $C_K$. 	We also enlarge $C_{K}$, if necessary, so that $C_{K}>2\ord_{\widetilde\frakP}(2)$ (since we applied Lemma~\ref{potmultred}). 

    \vspace{3mm}
    
    Now, we apply Lemma \ref{potmultred} again and see that the image of $\overline{\rho}_{E,p}$ contains an element of order $p$.  By Theorem \ref{subgroups} any subgroup of $\gl_2(\bbF_p)$ 
	having an element of order $p$ is either reducible or contains ${\rm SL}_2(\bbF_p)$. As $p>C_K>2\ord_{\frakP}(2)$, which implies that $\overline{\rho}_{E,p}$ is irreducible, the image	contains ${\rm SL}_2(\bbF_p)$.  Finally, we can ensure that $K\cap\bbQ(\zeta_p)=\bbQ$ by taking $C_K$ large enough if needed.	Hence, $\chi_\cyc=\det(\overline{\rho}_{E,p})$ is surjective giving the following short exact sequence
 \begin{equation*}
     1 \longrightarrow \SL_2(\bbF_p) \longrightarrow \overline{\rho}_{E,p}(G_K) \xrightarrow{\;\;\det\;\;} \bbF_p^\times \longrightarrow 1,
 \end{equation*}
 which completes the proof.
\end{proof}

\section{Level Lowering and Eichler-Shimura}

In this section, we will be relating the Galois representation $\overline{\rho}_{E,p}$ attached to the Frey curve $E$ with another representation of lower level.

\subsection{Totally real case}\label{eichlershimuratotallyrealcase} We present a level-lowering result by Freitas and Siksek \cite[Theorem 7]{fs15} derived from the work of Fujira, Jarvis, and Rajaei.

\begin{theorem}\label{level lowering}
 	Let $K$ be a totally real number field and $E/K$ an elliptic curve of conductor $\ccN_E$. Let $p$ be a rational prime. For a prime $\fraq$ of $K$, let $\Delta_\fraq$ denote the minimal discriminant of $E$ at $\fraq$, and $\ccN_p$ be as in \eqref{reducedconductor}.  Suppose that the following statements hold:
 	\begin{enumerate}[label=(\roman*)]
 		\item $p\geq 5$, the ramification index $e(\mathfrak{q}/p)<p-1$ for all $\mathfrak{q}\mid p$, and $\mathbb{Q}(\zeta_p)^+ \nsubseteq K$,
 		\item $E$ is modular,
 		\item $\overline{\rho}_{E,p}$ is irreducible,
 		\item $E$ is semistable at all $\mathfrak{q}\mid p$,
 		\item $p\mid \ord_{\mathfrak{q}}(\Delta_{\mathfrak{q}})$ for all $\mathfrak{q}\mid p$.
 	\end{enumerate}
 	Then there is a Hilbert eigenform $\mathfrak{f}$ of parallel weight $2$ that is new at level $\mathcal{N}_p$ and some prime $\varpi$ of $\bbQ_{\mathfrak{f}}$ such that $\varpi\mid p$ and 
 	$\overline{\rho}_{E,p} \sim \overline{\rho}_{\mathfrak{f},\varpi}$.
 \end{theorem}

  Freitas and Siksek obtain the following corollary from the above theorem in \cite{fs15} following the works of Blasius \cite{blas04}, Darmon \cite{darmon04} and Zhang \cite{zha01}.
 \begin{cor}\label{EScor}
     	Let $E$ be an elliptic curve over a totally real field $K$, and let $p$ be an odd prime. Suppose that
 	$\overline{\rho}_{E,p}$ is irreducible, and $\overline{\rho}_{E,p} \sim \overline{\rho}_{\mathfrak{f},p}$ for 
 	some Hilbert newform $\mathfrak{f}$ over $K$ of level $\ccN$ and parallel weight $2$ which satisfies $\bbQ_{\mathfrak{f}}=\mathbb{Q}$. 
 	Let $\mathfrak{q} \nmid p$ be a prime ideal of $\mathcal{O}_K$ such that:
 	\begin{enumerate}[label=(\roman*)]
 		\item $E$ has potentially multiplicative reduction at $\mathfrak{q}$,
 		\item $p\mid \#\overline{\rho}_{E,p}(I_{\mathfrak{q}})$,
 		\item $p \nmid (\norm_{K/\mathbb{Q}}(\mathfrak{q}) \pm 1)$
 	\end{enumerate}
 	Then there is an elliptic curve $E_{\mathfrak{f}}/K$ of conductor $\mathcal{N}$ having the same $L$-function as $\mathfrak{f}$.
 \end{cor}

 Let $K$ be a totally real number field, and fix $n\in\bbZ$ and $\widetilde\frakP\in S_K$. Let $(a,b,c)$ be a putative solution to the equation \eqref{mainequation} with $\widetilde\frakP\mid c$ and exponent $p$ such that $\ord_{\fraq}(n)<p$ for any prime $\fraq$ of $K$.

 \vspace{3mm}

First, it follows from  Lemma \ref{conductoroffreycurve} that $E$ is semistable outside $S_K$. Moreover, $\overline{\rho}_{E,p}$ is irreducible by Corollary~\ref{galrepsurjective}, and $E$ is modular by Proposition~\ref{modularityoffreycurveovertotallyreal} after taking $B_K$  sufficiently large. We then apply Theorem \ref{level lowering} (after enlarging $B_K$ to ensure that $\mathbb{Q}(\zeta_p)^+ \nsubseteq K$) and Lemma \ref{conductoroffreycurve} to obtain  $\overline{\rho}_{E,p}\sim \overline{\rho}_{\mathfrak{f},\varpi}$ for some Hilbert newform  $\fraf$ of level $\mathcal{N}_p$ and some prime $\varpi\mid p$ of $\bbQ_{\mathfrak{f}}$ where $\bbQ_\fraf$ denotes the field generated by the Hecke eigenvalues of $\fraf$. 

\vspace{3mm}

Now we reduce to the case where $\bbQ_{\mathfrak{f}}=\mathbb{Q}$, after possibly enlarging $B_K$ by an effective amount. This step uses standard ideas originally due to Mazur that can be found in \cite[Section 4]{bs04} and  \cite[Proposition 15.4.2]{cohen07}.  Next, we want to show that there is some elliptic curve $E'/K$ of conductor $\mathcal{N}_p$ having the same $L$-function as $\mathfrak{f}$.  We know that $E$ has potentially multiplicative reduction at $\widetilde\frakP$ and $p\mid \# \overline{\rho}_{E,p}(I_{\widetilde\frakP})$ by Lemma \ref{potmultred}.  We can conclude that there is an elliptic curve $E'=E_\fraf$ of conductor $\ccN_p$ satisfying $\overline{\rho}_{E,p} \sim \overline{\rho}_{E',p}$ if we implement Corollary \ref{EScor} after possibly enlarging $B_K$ to ensure that $p \nmid (\norm_{K/\mathbb{Q}}(\widetilde{\mathfrak{P}})\pm 1)$.

\subsection{General case} 

We we apply Conjecture~\ref{conj1}, we relate the residual representation $\overline{\rho}_{E,p}$ to a weight-two, mod $p$ eigenform $\theta$ over $K$ of level $\frakN_E$. One of the main steps toward the proof is to lift mod $p$ eigenforms to complex ones. The following result is due to \c{S}eng\"{u}n and Siksek.

\begin{prop}[\cite{ss18}, Proposition 2.1]\label{eigenform}
	There is an integer $B(\frakN)$ depending only on $\frakN$ such that for any prime $p>B(\frakN)$, every weight two, mod $p$ eigenform of level $\frakN$ lifts to a complex one.
\end{prop}

In addition to Conjecture~\ref{conj1}, we will use a special case of a fundamental conjecture from the Langlands Programme.

\begin{conj}[\cite{ss18}, Conjecture 4.1]\label{conj2}
	Let $\fraf$ be a weight 2 complex eigenform over $K$ of level $\frakN$ that is non-trivial and new.  If $K$ has some
	real place, then there exists an elliptic curve $E_{\fraf}/K$ of conductor $\frakN$ such that 
	\begin{equation}\label{c2eqn}
		\#E_{\fraf}(\ccO_K/\fraq)=1+\norm(\fraq)-\fraf(T_{\fraq})\quad\mbox{for all}\quad\fraq\;\nmid\;\frakN.
	\end{equation}
	If $K$ is totally complex, then there exists either an elliptic curve $E_{\fraf}$ of conductor $\frakN$ satisfying (\ref{c2eqn})
	or a fake elliptic curve 
	$A_{\fraf}/K$, of conductor $\frakN^2$, such that
	\begin{equation}
		\#A_{\fraf}(\ccO_K/\fraq)=(1+\norm(\fraq)-\fraf(T_{\fraq}))^2\quad\mbox{for all}\quad\fraq\;\nmid\;\frakN.
	\end{equation}
\end{conj}	

Let $K$ be a number field with at least one complex embedding. If $K$ is an imaginary quadratic field such that the Mordell-Weil group $X_0(15)(K)$ is finite, assume Conjecture~\ref{conj2}; otherwise, assume Conjecture~\ref{conj1} and Conjecture~\ref{conj2} hold true for $K$. 

\begin{lemma}\label{eichlershimura}
	There is a non-trivial, new (weight 2) complex eigenform $\fraf$ which has an associated elliptic curve  $E_{\fraf}/K$ of conductor $\frakN'$ dividing $\frakN_E$.
\end{lemma}
\begin{proof}
     The proof closely follows the ideas in \cite{ss18}. We first show the existence of such an eigenform $\fraf$ of level $\frakN_E$ dividing $\ccN_p$. By Corollary \ref{galrepsurjective}, we see that the representation $\overline{\rho}_{E,p}:G_K\rightarrow \gl_{2}(\mathbb{F}_p) $ is surjective, hence is absolutely irreducible for $p>C_K$.  Now, we apply  Conjecture \ref{conj1} to deduce that there is 
	a weight $2$ mod $p$ eigenform $\theta$ over $K$ of level $\frakN_E$, with $\frakN_E$ as in Lemma \ref{conductoroffreycurve}, such that for all primes $\fraq$ coprime to $p\frakN_E$, we have 	\[\tr(\overline{\rho}_{E,p}(\frob_{\fraq}))=\theta(T_{\fraq}).\]

We also know from the same lemma that there are only finitely many possible levels $\frakN_E$.  Thus by taking $p$
	large enough, see Proposition \ref{eigenform}, for any level $\frakN_E$
	there is a weight 2 complex eigenform $\fraf$ of level $\frakN_E$ which is a lift of $\theta$. Note that since there are only finitely many such eigenforms $\fraf$
	and they depend only on $K$, from now on we can suppose that every constant depending on these eigenforms depends
	only on $K$.

 \vspace{3mm}

 Next, following exactly the same steps as in \cite[Lemma 7.2]{ss18} we can see that if $\bbQ_\fraf\neq\bbQ$ then  there is a constant $C_{\fraf}$ depending only on $\fraf$ such that 
	$p<C_{\fraf}$.  Therefore, by taking $p$ sufficiently large, we assume that $\bbQ_\fraf=\bbQ$. 
	In order to apply Conjecture \ref{conj2}, we need to
	show that $\fraf$ is non-trivial and new.  As $\overline{\rho}_{E,p}$ is irreducible, the eigenform $\fraf$ is non-trivial.  
	If $\fraf$ is new, we are done.  If not, we can replace it with an equivalent new eigenform of a smaller level.  Therefore,
	we can take $\fraf$ new with level $\frakN'$ dividing $\frakN_E$.
	Finally, we apply Conjecture \ref{conj2} and obtain that $\fraf$ either has an associated elliptic curve $E_{\fraf}/K$ of conductor $\frakN'$, or has an associated fake elliptic curve $A_{\fraf}/K$ of conductor $\frakN_E^2$.

 \vspace{3mm}
	
	By Lemma \ref{fakeellipticcurve} below, if $p>24$, then $\fraf$ has an associated elliptic curve $E_{\fraf}$. As a result, we can assume that $\overline{\rho}_{E,p}\sim \overline{\rho}_{E',p}$ where $E'=E_{\fraf}$ is an elliptic curve with conductor $\frakN'$ dividing $\frakN_E$.  

\end{proof}

\begin{lemma}\label{fakeellipticcurve}
	If $p>24$, then $\fraf$ has an associated elliptic curve $E_{\fraf}$.
\end{lemma}
\begin{proof}
  The proof will follow closely \cite[Lemma 7.3]{ss18}. Assume for a contradiction that $\fraf$ has an associated fake elliptic curve $A_\fraf/K$. It then follows from \cite[Theorem~4.2 and Lemma~5.1 ]{ss18} that $\sharp\overline{\rho}_{A_\fraf,p}(I_{\widetilde\frakP})\leq 24$, where $\overline{\rho}_{A_\fraf,p}$ is the $2$-dimensional representation attached to $A_\fraf$ defined in \cite[\S 4]{ss18}. Since $p$ divides $\sharp\overline{\rho}_{E,p}(I_{\widetilde\frakP})$ by Lemma~\ref{potmultred} and $\overline{\rho}_{E,p}\sim \overline{\rho}_{A_\fraf,p}$, we obtain a contradiction.
\end{proof}

  \begin{lemma}\label{full2torsion}
  If $E'$ does not have full $2$-torsion, and is not $2$-isogenous to an elliptic
curve with full $2$-torsion, then $p<C_{E'}$ where $C_{E'}$ is a constant depending only on $E'$. 
		\end{lemma}
  \begin{proof}
      By \cite[Theorem 2]{Katz81} there are infinitely many primes $\fraq$ such that $\# E'(\bbF_\fraq)\not \equiv 0 \pmod 4$.  Fix such a prime $\fraq\not\in S_K$ and note that $E$ is semistable at $\fraq$.  If $E$ has good reduction at $\fraq$, then $\# E(\bbF_\fraq)\equiv \# E'(\bbF_\fraq) \pmod p$. Note that $\sharp E(\bbF_\fraq)$ is divisible by $4$ as the Frey curve $E$ has full $2$-torsion  Therefore, the difference $\# E(\bbF_\fraq)- \# E'(\bbF_\fraq)$, which is divisible by $p$, is nonzero.  As the difference belongs to a finite set depending on $\fraq$, $p$ becomes bounded.  If $E$ has multiplicative reduction at $\fraq$, we obtain 
      \[
		\pm(\norm(\fraq)+1)\equiv a_{\fraq}(E') \pmod p
		\]
		by comparing the traces of Frobenius.  We see that this difference being also nonzero and depending only on $\fraq$ gives a bound for $p$.
  \end{proof}

   \begin{theorem}\label{levellowerinandeichlershimura}
Let $K$ be a number field, and let $\widetilde\frakP\in S_K$ be a fixed prime. If $K$ does not satisfy \ref{itm:HNf}, we assume Conjecture~\ref{conj1} and Conjecture~\ref{conj2} hold for $K$. There is a constant $B_K$ depending only on $K$ such that the following hold. 
Let $(a,b,c)$ be a putative solution to the equation \eqref{mainequation} with $\widetilde\frakP\mid c$ and exponent $p$ such that $\ord_{\fraq}(n)<p$ for any prime $\fraq$ of $K$. Write $E$ for the Frey curve in \eqref{freycurve2}.
Then there is an elliptic curve $E'$ over $K$ such that:
\begin{enumerate}[label=(\roman*)]
\item the elliptic curve $E'$ has conductor $\frakN'$ dividing $\ccN_p$;
\item $\overline{\rho}_{E,p} \sim \overline{\rho}_{E',p}$;
\item $E'$ has full $2$-torsion over $K$;
\item for $\widetilde\frakP\in S_K$, $\ord_{\widetilde\frakP}(j_{E'})<0$ where $j_{E'}$ is the $j$-invariant of $E'$.
\end{enumerate}
\end{theorem}  
\begin{proof}
    It follows from the argument at the end of Section~\ref{eichlershimuratotallyrealcase} in the totally real case, or Lemma~\ref{eichlershimura} in the general case that, for $p$ large enough, we have an elliptic curve $E'/K$ of conductor $\frakN'$ dividing $\ccN_p$ such that $\overline{\rho}_{E,p}\sim \overline{\rho}_{E',p}$, which completes (i) and (ii) as $\frakN_E$ is supported on $S_K$ by Lemma~\ref{conductoroffreycurve}. 
    
\vspace{3mm}
    
    If $E'$ does not have full $2$-torsion over $K$, then it is $2$-isogenous to an elliptic curve $E''$ having full $2$-torsion over $K$ (after enlarging $p$).  As the isogeny induces an isomorphism $E'[p]\cong E''[p]$ of Galois modules ($p\neq 2$), we get $\overline{\rho}_{E,p}\sim \overline{\rho}_{E',p}\sim\overline{\rho}_{E'',p}$.  After possibly replacing $E'$ by $E''$, we can suppose that $E'$ has full $2$-torsion over $K$ giving us (iii).

    \vspace{3mm}

    It remains to prove  $v_{\widetilde\frakP}(j')<0$ where $j'$ is the $j$-invariant of $E'$. By Lemma~\ref{potmultred}, the prime $p$ divides the size of $\overline{\rho}_{E,p}(I_{\widetilde\frakP})$. 
	It now follows from Lemma~\ref{imageofinertia} that $v_{\widetilde\frakP}(j')<0$ since the sizes of  
	$\overline{\rho}_{E,p}(I_{\widetilde\frakP})$ and $\overline{\rho}_{E',p}(I_{\widetilde\frakP})$ are equal.
\end{proof}

\section{Proof of Main Theorems}

In this section, we will handle the last step of the modular method and establish the proofs of the statements given in Section~\ref{statementsoftheresults}. 

\subsection{$S$-unit equations and elliptic curves} Let $K$ be a number field, $\ccO_K$ be its ring of integers, and $S$ be a finite set of prime ideals of $\ccO_K$.
\begin{definition}
    An $S$\textit{-unit} is an element $\alpha\in K$ such that the principal fractional ideal generated by $\alpha$ can be written as a product of the prime ideals in $S$. Namely, the set of $S$-units $\ccO_S^\times$ can be defined as:
    \begin{equation*}
        \ccO_S^\times=\{\alpha\in K^\times\;:\; \ord_\frap(\alpha)=0 \; \text{for all}\; \frap\not\in S \}.
    \end{equation*}
\end{definition}

\vspace{2mm}

\begin{definition}
    Let $K$ be a number field and $S$ a finite set of prime ideals of $\ccO_K$. The $S$\textit{-unit equation} is the equation
    \begin{equation}\label{sunitequation}
        \lambda+\mu=1,\; \lambda,\mu\in \ccO_S^\times.
    \end{equation}
    If $S=\emptyset$, then $\ccO_S^\times=\ccO_K^\times$, so this equation is called the \textit{unit equation}.
\end{definition}

\vspace{2mm}

 Let $\Lambda_S=\{(\lambda,\mu): \lambda+\mu=1,\; \lambda,\mu\in \ccO_S^\times \}$. Observe that the $S$-unit equation \eqref{sunitequation} has trivial set of solution $\{(-1,2),(2,-1),(1/2,1/2)\}$. Following \cite{freitassiksek15} these solutions will be called \textit{irrelevant} solutions, while other solutions will be called \textit{relevant} solutions. 

\vspace{3mm}

\begin{theorem}[Siegel 1921, Parry 1950]
    Let $K$ be a number field and $S$ a finite set of prime ideals of $\ccO_K$. The $S$-unit equation
    \begin{equation*}
       \lambda+\mu=1,\; \lambda,\mu\in \ccO_S^\times.
    \end{equation*}
has only finitely many solutions.
\end{theorem}

Note that the original proofs due to Siegel and Parry are non-effective. Baker’s theory of linear forms in logarithms gave effective bounds although they are very large for the solutions. Various algorithms and bounds can be found in the literature, see for example \cite{dweg89,smart99,akm21,km23}.

\vspace{3mm}

Before going into the details of the proofs, we will state the relationship between solutions to $S$-unit equations and certain families of elliptic curves. Consider an elliptic curve $E'$ with full $2$-torsion over $K$ having potentially good reduction outside $S$.  It then follows that it has the form
\begin{equation*}
     Y^2=(X-e_1)(X-e_2)(X-e_3),
\end{equation*}
where $e_1,e_2,e_3\in K$ are distinct.  Let $ \lambda:=\frac{e_3-e_1}{e_2-e_1} $ be the $\lambda$-invariant of the elliptic curve (see \cite[III.1]{sil09}). Write the \textit{Legendre form} of $E'$
\begin{equation*}
    E_\lambda : Y^2=X(X-1)(X-\lambda).
\end{equation*}
 It then follows from \cite[III.1]{sil09} that the $j$-invariant of $E_\lambda$ is
 \begin{equation}\label{j-invpp2}
 j(\lambda)=2^8\frac{(\lambda^2-\lambda+1)^3}{\lambda^2(1-\lambda)^2}.
 \end{equation}
Note that $\lambda\in K$ as $E'$ has full $2$-torsion over $K$. Let  $\frakS_3$ denote the symmetric group on three letters $\{e_1,e_2,e_3\}$. The action of $\frakS_3$ on the set $\{e_1,e_2,e_3\}$ can be extended to an action on $\bbP^1(K)-\{0,1,\infty\}$ via the ratio $ \lambda=\frac{e_3-e_1}{e_2-e_1} $. Under the action of $\frakS_3$, the orbit of $\lambda$ in $\bbP^1(K)-\{0,1,\infty\}$ is the set:
 \begin{equation}\label{lambdainvariants}
M_\lambda:= \left\{\lambda,\frac{1}{\lambda},1-\lambda,\frac{1}{1-\lambda},\frac{\lambda}{\lambda-1},\frac{\lambda-1}{\lambda}\right\}.
 \end{equation}

  Moreover, there is a one-to-one correspondence between the $\overline{K}$-isomorphism classes of the elliptic curves with full $2$-torsion over $K$ and the $\frakS_3$-orbits of elements of $\lambda\in\bbP^1(K)-\{0,1,\infty\}$ (see \cite[Lemma 5.2]{OS22}). 

  \vspace{3mm}

 As $E'$ has potentially good reduction away from $S$, the $j$-invariant $j_{E'}$ belongs to $\ccO_{S}$.  From the equation \eqref{j-invpp2}, we can deduce that $\lambda\in K$ satisfies a degree six monic polynomial with coefficients in $\ccO_{S}$ implying $\lambda\in \ccO_{S}$.  On the  other hand, notice that $1/\lambda$  and $1/\mu$ (with $\mu:=1-\lambda$)  are solutions of \eqref{j-invpp2} and hence elements of  $\ccO_{S}$.  Therefore, we conclude that $(\lambda,\mu)$ is a solution of the $S$-unit equation \eqref{sunitequation}.

\vspace{3mm}

\subsection{Proof of Theorem~\ref{maintheoremA}} 
 
Let $K$ be a number field. If $K$ does not satisfy \ref{itm:HNf}, we assume Conjecture~\ref{conj1} and Conjecture~\ref{conj2} hold for $K$. We have so far proved that a putative non-trivial solution $(a,b,c)$ to the equation \eqref{mainequation}  with $\widetilde{\frakP}\mid c$ and exponent $p$ such that $\ord_{\fraq}(n)<p$ for any prime $\fraq$ of $K$ yields an elliptic curve $E'/K$ with full $2$-torsion over $K$ having potentially good reduction outside $S$ by Theorem~\ref{levellowerinandeichlershimura}. Recall that we set
\begin{equation*}
    S=\{\frakP\;:\; \frakP \text{ is a prime of }K\text{ above }2\}\cup \{\fraq: \fraq \text{ is an odd prime of }K\text{ dividing }n\}.
\end{equation*} 

 \vspace{3mm}
 
  Suppose now, as in the statement of Theorem~\ref{maintheoremA}, that the prime $\widetilde{\frakP}\in S_K\subset S$ satisfies the following: for every solution $(\lambda,\mu)\in \left(\ccO_{S}^\times \right)^2$ to the equation
    \begin{equation*}
        \lambda+\mu=1
    \end{equation*}    
we have $\max\{| \ord_{\widetilde{\frakP}}(\lambda)| , | \ord_{\widetilde{\frakP}}(\mu) | \}\leq 4\ord_{\widetilde{\frakP}}(2)$. Let $m:=\max\{| v_{\widetilde{\frakP}}(\lambda)| ,| v_{\widetilde{\frakP}}(\mu)| \}$ and further let us express $j_{E'}$ in terms of $\lambda$  and $\mu$ as:
 \begin{equation}\label{newj-invariant}
 j_{E'}=2^8\frac{(1-\lambda\mu)^3}{(\lambda\mu)^2}.
 \end{equation}
 If $m=0$, then $v_{\widetilde{\frakP}}(j_{E'})\geq 8v_{\widetilde{\frakP}}(2)\geq0$ by \eqref{newj-invariant}. This contradicts with the  assumption that $v_{\widetilde{\frakP}}(j_{E'})<0$ by Theorem~\ref{levellowerinandeichlershimura}.  Hence, we may suppose $m>0$.  The relation
 $\lambda+\mu=1$ leads to $v_{\widetilde{\frakP}}(\lambda+\mu)\geq \min\{v_{\widetilde{\frakP}}(\lambda),v_{\widetilde{\frakP}}(\mu)\}$ with equality
 if $v_{\widetilde{\frakP}}(\lambda)\neq v_{\widetilde{\frakP}}(\mu)$. This shows either $v_{\widetilde{\frakP}}(\lambda)=v_{\widetilde{\frakP}}(\mu)=-m$, or  $v_{\widetilde{\frakP}}(\lambda)=m$ and $v_{\widetilde{\frakP}}(\mu)=0$, or $v_{\widetilde{\frakP}}(\lambda)=0$ and $v_{\widetilde{\frakP}}(\mu)=m$.  Therefore,  $v_{\widetilde{\frakP}}(\lambda\mu)=-2m<0$ or $v_{\widetilde{\frakP}}(\lambda\mu)=m>0$.  In any of the cases, we have  $v_{\widetilde{\frakP}}(j_{E'})\geq 8v_{\widetilde{\frakP}}(2)-2\ord_{\widetilde\frakP}(\lambda\mu)\geq 0$ as $m\leq 4\ord_{\widetilde\frakP}(2)$, which again yields a contradiction.  

\vspace{3mm}

\subsection{Proof of Corollary~\ref{corollaryoftheoremAquadraticfields}}

 Let $d\in \bbZ$ be a squarefree integer, and $K=\bbQ(\sqrt{d})$. Assume that we have one of the following conditions:
    \begin{enumerate}[label=(\roman*)]
          
        \item If $d>0$, i.e. $K$ is a real quadratic field, then $d\geq 13$ and $d\equiv 5 \pmod{8}$.
        \item If $d<0$, i.e. $K$ is an imaginary quadratic field, then $| d | \geq 7$ and $d \equiv 2,3 \pmod 4$. Assume also that Conjecture~\ref{conj1} and Conjecture~\ref{conj2} hold for $K$.
    \end{enumerate}
Let $\ell\geq 29$ be a prime satisfying $\ell \equiv 5 \pmod 8$ and $\left(\frac{d}{\ell} \right)=-1$.

\subsubsection{Proof of Case (i)}
 Let $K=\bbQ(\sqrt{d})$ be a real quadratic field. It then follows that both $2$ and $\ell$ are inert in $K$, so $S=\{2,\ell\}$. It follows from Theorem~\ref{levellowerinandeichlershimura} and the argument above that a putative solution $(a,b,c)$ such that $2\mid c$ yields an $S$-unit solution $(\lambda,\mu)$ with $S=\{2,\ell\}$. The result now follows from the lemma below.
 
 \begin{lemma}[Lemma 7.5, \cite{dec16}]
    Let $d\geq 13$ be a squarefree integer that satisfies $d\equiv 5 \pmod{8}$ and let $\ell\geq 29$ be a prime satisfying $\ell \equiv 5 \pmod 8$ and $\left(\frac{d}{\ell} \right)=-1$. Let $K=\bbQ(\sqrt{d})$ and $S=\{2,\ell\}$. Then there are no relevant elements of $\Lambda_S$.
\end{lemma}

\subsubsection{Proof of Case (ii)}

Let $K = \bbQ(\sqrt{-d})$ be an imaginary quadratic field. In this case, we have $S=\{\frakP,\ell\}$, where $2\OK=\frakP^2$ and $\frakP$ is not principal. We then deduce by that Theorem~\ref{levellowerinandeichlershimura} and the argument above that a putative solution $(a,b,c)$ such that $2\mid c$ yields an $S$-unit solution $(\lambda,\mu)$ with $S=\{\frakP,\ell\}$. The result now follows from the lemma below.

\begin{lemma}[Lemma 7.3, \cite{KO20}]
     Let $K = \bbQ(\sqrt{-d})$ be an imaginary quadratic field, where $d\geq 7$ is a squarefree positive integer such that $-d \equiv 2,3 \pmod 4$. Let $\ell\geq 29$ be a prime satisfying $\ell \equiv 5 \pmod 8$ and $\left(\frac{d}{\ell} \right)=-1$.  Then there are no relevant elements of $\Lambda_S$.
\end{lemma}

    \vspace{3mm}

 \subsection{Proof of Corollary~\ref{maintheoremptotallyramifies}} Assume that $K$ and $p$ satisfy the hypotheses of Corollary~\ref{maintheoremptotallyramifies}. Namely, $K$ is a field of degree $d$ such that $2$ totally ramifies in $K$, and $p \leq 5$ is a prime totally ramified in $K$ and satisfying $\gcd(d,p-1)=1$. As before we take $\frakP$ to be the unique prime above $2$ and $S = \{\frakP\}$. We then deduce by Theorem~\ref{levellowerinandeichlershimura} and the argument above that a putative solution $(a,b,c)$ such that $\frakP\mid c$ yields an $S$-unit solution $(\lambda,\mu)$, which satisfies $ m_{\lambda,\mu} < 2 \ord_\frakP(2)$ by Lemma~\ref{lemmaC}.

\begin{lemma}[Lemma 3.1, \cite{FKS21}]\label{lemmaC}
     Assume that $K$ and $p$ satisfy the hypotheses of Corollary~\ref{maintheoremptotallyramifies}. Every solution $(\lambda, \mu)$ to the $S$-unit equation satisfies
        \begin{equation*}
             m_{\lambda,\mu}=\max\{|\ord_\frakP(\lambda | , | \ord_\frakP(\mu)| \}< 2 \ord_\frakP(2).
        \end{equation*}
    \end{lemma}

\vspace{2mm}

\subsection{Proof of Corollary~\ref{maintheorem3totallysplits}}  Let $K$ be a field of odd degree $d$, and suppose that $2$ is totally ramified in $K$ and that $3$ totally splits in $K$. If $K$ does not satisfy \ref{itm:HNf}, we assume Conjecture~\ref{conj1} and Conjecture~\ref{conj2} hold for $K$. We claim that every solution to $(\lambda, \mu)$ to the $S$-unit equation \eqref{unitequation} satisfies $ m_{\lambda, \mu}<2 \ord_\frakP(2)$. Our claim combined with Theorem~\ref{levellowerinandeichlershimura} immediately implies Corollary~\ref{maintheorem3totallysplits}.

\vspace{3mm}

Suppose $(\lambda, \mu)$ is a solution to the $S$-unit equation \eqref{unitequation} with $ m_{\lambda, \mu}\geq2 \ord_\frakP(2)$. By Lemma~\ref{lemmaA}, there is a solution $(\lambda', \mu')$ to \eqref{unitequation} such that $\lambda'\in \ccO_K , \mu'\in \ccO_K^\times$, and $\ord_\frakP(\lambda')=m_{\lambda, \mu}=m_{\lambda',\mu'}\geq 2\ord_\frakP(2)$. It follows that  $\mu' \equiv 1 - \lambda' \equiv 1 \pmod 4$, whence $\norm_{K/\bbQ}(\mu')=1$. However, by Lemma~\ref{lemmaD} we have $\mu' \equiv -1 \pmod{3}$ and so $\norm_{K/\bbQ}(\mu') =(-1)^d = -1$ since $d$ is odd. This gives a contradiction, and completes the proof of Corollary~\ref{maintheorem3totallysplits}.

\vspace{3mm}

\begin{lemma}[Lemma 2.1, \cite{FKS21}]\label{lemmaA}
    Let $K$ be a number field in which $2$ is either inert or totally ramified, say $S=\{\frakP\}$. Let $(\lambda, \mu)$ be a solution to the $S$-unit equation \eqref{unitequation}, and write
    \begin{equation*}
        m_{\lambda,\mu}=\max\{| \ord_\frakP(\lambda)| , | \ord_\frakP(\mu)| \}.
    \end{equation*}
    Then there is a solution $(\lambda', \mu')$ to the $S$-unit equation with 
    \begin{equation*}
        \lambda'\in \ccO_S\cap \ccO_K^\times,\; \mu'\in \ccO_K^\times,\; \text{and}\;  m_{\lambda,\mu}= m_{\lambda',\mu'}.
    \end{equation*}
    \end{lemma}

    \vspace{3mm}

     \begin{lemma}[Lemma 5.1,\cite{FKS21}]\label{lemmaD}
        Let $K$ be a number field in which $3$ splits completely. Let $S$ be a finite set of primes of $K$ that is disjoint from the primes above $3$. Let $(\lambda, \mu)$ to the $S$-unit equation with $\lambda,\mu\in\ccO_K$. Then
        \begin{equation*}
            \lambda\equiv \mu \equiv -1 \pmod 3.
        \end{equation*}
    \end{lemma}

\vspace{3mm}

\subsection{Proof of Corollary~\ref{realquadraticwithequiv24}}  Let $q>73$ be a prime such that $q \equiv 1 \pmod{24}$, and $K = \bbQ(\sqrt{q})$. Write $\frakP_1, \frakP_2$ for the two primes of $K$ above $2$ and let $S = \{\frakP_1, \frakP_2\}$.  The result now follows from Lemma~\ref{lemma11.2} and Theorem~\ref{levellowerinandeichlershimura} combined.

\begin{lemma}[Lemma 11.2, \cite{FKS}]\label{lemma11.2}
    Let $q \equiv 1 \pmod{24}$ be prime, and $K = \bbQ(\sqrt{q})$. Write $\frakP_1, \frakP_2$ for the two primes of $K$ above $2$ and let $S = \{\frakP_1, \frakP_2\}$. If $q>73$ then the solutions $(\lambda,\mu)$ to
the $S$-unit equation \eqref{unitequation} satisfy $\max\{\ord_{\frakP_1}(\lambda),\ord_{\frakP_1}(\mu)\}=1$ or $\max\{\ord_{\frakP_1}(\lambda),\ord_{\frakP_1}(\mu)\}=1$.
\end{lemma}

\subsection{Proof of Theorem~\ref{densitytheoremforourcase}} If $d\in \ccD$ (resp. $d\in\ccD'$), let $K=\bbQ(\sqrt{d})$ (resp. $K=\bbQ(\sqrt{-d})$). It follows from Theorem~\ref{levellowerinandeichlershimura} and the argument above that a putative solution $(a,b,c)$ such that $2\mid c$ yields an $S$-unit solution $(\lambda, \mu)$ with $S=S_K$. The result now follows from Theorem~\ref{densitytheorem}.

\vspace{3mm}

\subsection{Proof of Theorem~\ref{maintheoremB}} Let $K$ be a number field as in Theorem~\ref{maintheoremB}. Additionally, assume that $K$ has a unique prime $\frakP$ above $2$, i.e. $S=S_K=\{\frakP\}$. Then, by Theorem~\ref{levellowerinandeichlershimura}, there exists a constant $B_K$ that depends only on $K$ such that a putative non-trivial solution $(a,b,c)$ with $\frakP\mid b$ to the equation \eqref{mainequation} with exponent $p>B_K$ gives rise to an elliptic curve $E'/K$ with a $K$-rational $2$-isogeny, good reduction away from $\frakP$ and potentially multiplicative reduction at $\frakP$. But this contradicts with Theorem \ref{thm for tot ram} below applied with $q=2$, hence there is no such a solution.

\begin{theorem}[\cite{FKS}, Theorem 1]\label{thm for tot ram}
	Let $q$ be a rational prime. Let $K$ be a number field satisfying the following conditions:
	\begin{itemize}
		\item $\bbQ(\zeta_{q}) \subset K$, where $\zeta_{q}$ is a primitive $q^{th}$ root of unity;
		\item $K$ has a unique prime $\fraq$ above $q$;
		\item ${\rm gcd}(h_K^+,q(q-1))=1$, where $h_K^+$ is the narrow class number of $K$.
	\end{itemize}
	Then there is no elliptic curve $E/K$ with a $K$-rational $q$-isogeny, good reduction away from $\fraq$, potentially multiplicative reduction at $\fraq$.
\end{theorem}

\bibliographystyle{alpha} 
\bibliography{reference}

\end{document}